\documentclass[12pt]{amsart}

\usepackage{amsmath,amssymb,amsthm}

\usepackage{amsmath,amstext,amssymb,amsopn,amsthm}
\usepackage{hyperref}
\usepackage{mathtools}
\usepackage{enumerate,bbm}

\usepackage{color,graphicx}

\usepackage[margin=30mm]{geometry}

\usepackage{graphicx,tikz, caption}

\allowdisplaybreaks

\newtheorem{theorem}{Theorem}[section]

\newtheorem{proposition}[theorem]{Proposition}
\newtheorem{corollary}[theorem]{Corollary}

\newtheorem{lemma}[theorem]{Lemma}

\theoremstyle{definition}

\newtheorem{example}[theorem]{Example}

\theoremstyle{remark}
\newtheorem{remark}[theorem]{Remark}

\newcommand{\prt}{\partial}

\newcommand{\vphi}{\varphi}

\newcommand{\R}{\mathbb{R}}
\newcommand{\Z}{\mathbb{Z}}

\newcommand{\calA}{\mathcal{A}}

\numberwithin{equation}{section}

\begin{document}

\title[Functional equation]{Freezing in space-time:\\
A functional equation linked with a PDE system}

\author{Krzysztof Burdzy and Adam J. Ostaszewski}

\address{KB: Department of Mathematics, Box 354350, University of Washington, Seattle, WA 98195, USA}
\email{burdzy@uw.edu}

\address{AJ: Mathematics Department, London School of Economics, Houghton Street, London WC2A 2AE, UK}
\email{a.j.ostaszewski@lse.ac.uk}

\thanks{KB's research was supported in part by Simons Foundation Grant 506732. }



\pagestyle{headings}

\begin{abstract}

We analyze the functional equation $F(x+F(x))=-F(x)$ and reveal its relationship with a system of partial differential equations arising as the hydrodynamic limit of a system of pinned billiard balls on the line. The system of balls must freeze at some time, i.e., no velocity may change after the freezing time.
The terminal velocity
and the freezing time profiles play the role of boundary conditions for the PDEs (qua terminal conditions, despite being ``initial conditions'' from the wave equation perspective pursued here).
Solutions to the functional equation
provide the link between the freezing time and terminal velocity profiles on the one hand, and the solution to the PDE in the entirety of the space-time domain on the other.
\end{abstract}

\maketitle

\section{Introduction}

This article is concerned with the functional equation
\begin{align}\label{m12.1}
F(x+F(x))=-F(x),
\end{align}
where $F$ represents an unknown real-valued function defined on $\R$ or a subset of $\R$,
and its key role in linking the solutions of a PDE system to its boundary conditions.
In view of its structure \eqref{m12.1} is termed a single-variable, composite, functional equation (cf. \cite{Aczel89} Chap. 19). Nearest in structure  to \eqref{m12.1}, but in two variables,
is the  Go\l \c{a}b-Schinzel equation
\begin{align}\label{m12.3}
F(x+yF(x))=F(x)F(y).
\end{align}
Its continuous solutions take one of the three forms: $F(x)\equiv 0$,  $F_{\rho}(x):=1+\rho x$, or $F(x)=\max\{F_{\rho}(x),0\}$, for some constant $\rho\in \R$ (see \cite[p. 17]{Brz}). 
For a further discussion of \eqref{m12.3} see \cite[Ch. 19]{Aczel89} and for a survey of the literature \cite{Brz}. Since $F_{-2}(1)=-1$, substituting $y=1$ in \eqref{m12.3} shows that $F_{-2}$ is a solution to \eqref{m12.1}; cf. Example \ref{m23.22}. However, the solutions to \eqref{m12.1} are more plentiful.

Equation \eqref{m12.1} arises in the analysis of a system of ``pinned balls'' and the corresponding partial differential equations. We will discuss this application of \eqref{m12.1} in Section \ref{sec:wave}. See Remark \ref{a1.1} for a brief outline of the relationship between pinned balls and partial differential equations.
The differential equations discussed in this article are amenable to analysis using well-known methods. A more complete and more general analysis of our system of PDEs will be presented in a separate paper.

There are numerous books devoted to functional equations, see for example, \cite{Aczel66,Aczel89,Small}. We did not find equation \eqref{m12.1} in any of these books or in other publications.

Equation \eqref{m12.1} is closely related to involutions. We will use results on involutions presented in \cite{Zamp}. We refer the reader to that article for a review of the literature, noting here only the early contribution \cite{Aczel48}.

Equation \eqref{m12.1} and associated involutions will be discussed in Section \ref{feinv}. This will be followed by Section \ref{examp} with examples. Finally, Section \ref{sec:wave} will present an application of \eqref{m12.1} to partial differential equations, including examples.

\section{The functional equation \eqref{m12.1} and involutions}\label{feinv}

We will discuss solutions $F$ to \eqref{m12.1} defined on a set $A\subset \R$.
Note that $F(x) \equiv 0$ is a solution to \eqref{m12.1} on $\R$ so the set of  solutions is non-empty.

Let
\begin{align*}
\vphi(x) = x+ F(x), \qquad x\in A.
\end{align*}
The dependence of the function $\vphi$ on $F$ is not reflected in the notation but this should not cause confusion.

\begin{proposition}\label{m19.1}
Suppose that $F$ is  continuous and solves \eqref{m12.1} on an interval $I$, and $I$ is a maximal interval in this respect, i.e., if $F$ is continuous and solves \eqref{m12.1} on an interval $I'$ such that $I\subset I'$ then $I'=I$. Note that $F$ must necessarily be defined on $J:= \vphi(I)$ for \eqref{m12.1} to be meaningful.

(i) The function $\vphi$ is an involution on $I$, i.e.,
\begin{align*}
\vphi(\vphi(x)) = x , \qquad x\in I.
\end{align*}

(ii) $\vphi(J)=I$.

(iii) The function $\vphi$ is continuous and strictly monotone on $I$.

(iv) The function $\vphi$ is continuous and strictly monotone on $J$.

(v) If $I\cup J$ is an interval then $\vphi$ is continuous and strictly decreasing on $J$.
\end{proposition}

\begin{proof}
 Consider any $x\in I$ and let $z=\varphi (x)=x+F(x)$. Then, using \eqref{m12.1},
\[
\varphi (\varphi (x))=\varphi (z)=z+F(z)=x+F(x)+F(x+F(x))=x+F(x)-F(x)=x.
\]
This proves (i). Claim (ii) follows from (i).

(iii)
If $\varphi (x)=\varphi (y)$ then $x=\varphi (\varphi (x))=\varphi
(\varphi (y))=y$. In other words, $\vphi$ is injective on $I$. Since $F$ is continuous on $I$, so is $\vphi$. This implies that $\vphi$ is continuous and strictly monotone on $I$.

(iv) The graphs $\{(x,\vphi(x)), x\in I\}$ and $\{(x,\vphi(x)), x\in J\}$ are symmetric with respect to the diagonal $\{(x,x), x\in \R\}$. This and (iii) imply (iv).

(v) We copy an argument from \cite[Prop. 1.1]{Zamp}.
If $I\cup J$ is an interval then $\vphi$ is continuous and strictly monotone on $I\cup J$, by (iii) and (iv).
Suppose that $\vphi$ is strictly increasing.
If $x\in I\cup J$ and $x < \vphi(x)$ then $\vphi(x) < \vphi(\vphi(x))=x$, a contradiction;
similarly, $x > \vphi(x)$ implies $\vphi(x) > x$.
\end{proof}

\begin{remark}\label{m23.20}
(i)
Proposition \ref{m19.1} applies to open, half-open, closed, finite, semi-infinite and infinite intervals $I$, including intervals consisting of a single point. In the case of a semi-infinite interval $I=(a,\infty)$, note that $F(x)$ is asymptotic on
that interval to $-x$, i.e., $\lim_{x\to \infty} F(x) +x=0$, if and only if $\lim_{x\to \infty} \varphi (x)= 0$, since $F(x)=-x+\varphi(x)$. A similar remark applies to $I=(-\infty,a)$.

(ii) Suppose that  $I\cup J$ is an interval. Then $\vphi$ is continuous and strictly decreasing on $I\cup J$. Since  $\{(x,\vphi(x)), x\in I\}$ and $\{(x,\vphi(x)), x\in J\}$ are symmetric with respect to the diagonal $\{(x,x), x\in \R\}$, there exists a unique $x_0$ such that $\vphi(x_0) = x_0$ and, therefore, $F(x_0) =0$. The function $F$ is strictly negative on the part of $I\cup J$ to the right of $x_0$ and it is strictly positive to the left of $x_0$.
\end{remark}

\begin{corollary}\label{m19.2}
The domain $A$ of a solution $F$ to \eqref{m12.1} can be decomposed into a family $\calA$ of disjoint sets of the form $I_\alpha\cup J_\alpha$, where the pair $(I_\alpha,J_\alpha)$ satisfies the conditions in Proposition \ref{m19.1}.
\end{corollary}

\begin{remark}
The set $\calA$ in Corollary \ref{m19.2} can be finite, countably infinite or uncountable. The sets $I_\alpha\cup J_\alpha$ may be intervals, pairs of disjoint intervals, single points or pairs of points.
\end{remark}

\begin{proposition}\label{m19.3}
Consider intervals $I,J\subset \R$ and a function $\vphi$ such that $\vphi: I \to J$, $\vphi: J\to I$, $\vphi$ is  continuous and strictly monotone on $I$, and is an involution on $I\cup J$. Then $F(x) := \vphi(x) -x$ is continuous on $I\cup J$ and solves \eqref{m12.1} on $I\cup J$.
\end{proposition}

\begin{proof}
The graphs $\{(x,\vphi(x)), x\in I\}$ and $\{(x,\vphi(x)), x\in J\}$ are symmetric with respect to the diagonal $\{(x,x), x\in \R\}$ because $\vphi$ is an involution. Since $\vphi$ is  continuous and strictly monotone on $I$, it follows that $\vphi$ is  continuous and strictly monotone on $J$.

Since $\vphi$ is continuous on $I\cup J$, so is $F$.

 Consider any $x\in I\cup J$ and let $z=\vphi (x)=x+F(x)$. Then
\begin{align*}
x+F(x)-F(x)=x = \vphi (\vphi (x))=\vphi (z)=z+F(z)=x+F(x)+F(x+F(x)).
\end{align*}
Subtracting $x+F(x)$ from both sides, we obtain $-F(x)=F(x+F(x))$, i.e., $F$ solves \eqref{m12.1} for $x\in I\cup J$.
\end{proof}

\begin{lemma}\label{m19.4}
 (Conjugacy, reflection and shift invariance)
Suppose that $F$ is a solution to \eqref{m12.1} on $A$.
For a fixed $a\ne 0$, let $F_{a}^{\times}(x):=a^{-1}F(ax)$ for $x$ such that $x\in A/a:=\{y: ya\in A\}$.
Let $F_{a}^{+}(x):=F(a+x)$ for $x$ such that $x\in A-a$. Then $F_a^\times$ and $F_a^+$ solve \eqref{m12.1} on $A/a$ and $A-a$, resp. In particular, the ``reflection'' $F^-(x):=-F(-x)$ is a solution on $-A$.

\end{lemma}

\begin{proof} Putting $ax$ for $x$ in \eqref{m12.1} yields the following
\begin{eqnarray*}
F(a[x+a^{-1}F(ax)]))+F(ax) &=&0, \\
a^{-1}F(a[x+a^{-1}F(ax)]))+a^{-1}F(ax) &=&0, \\
F_a^\times(x+F_a^\times(x))+F_a^\times(x) &=&0.
\end{eqnarray*}%
Likewise putting $a+x$ for $x$ yields%
\begin{eqnarray*}
F(a+x+F(a+x))+F(a+x) &=&0, \\
F_{a}^+(x+F_{a}^+(x))+F_{a}^+(x) &=&0.
\end{eqnarray*}
\end{proof}

The following lemma is a slightly modified \cite[Thm. 2.1]{Zamp}.

\begin{lemma}\label{m19.5}
Suppose that $a>0$ and $\psi:[-a,a] \to \R$ is  even, differentiable and $\left|\frac\prt{\prt x} \psi(x)\right| < 1$ for all $x$. Let $u\to (u+\psi(u))^{-1}$ denote the inverse of the function $x \to x+\psi(x)$.
Then
\begin{align}
\vphi(u) &= u- 2( u+\psi( u))^{-1}\label{m23.7}
\end{align}
is an involution on the interval $[-a+\psi(a), a+\psi(a)]$.
\end{lemma}

\begin{proof}
The  condition $\left|\frac\prt{\prt x} \psi(x)\right| < 1$ implies that if we rotate the graph $\Gamma_1$ of $\psi$ by $\pi/4$ clockwise and stretch it by the factor of $\sqrt{2}$ (in all directions) then it will be the graph $\Gamma_2$ of a function $\vphi$ on the interval $[-a+\psi(-a), a+\psi(a)] = [-a+\psi(a), a+\psi(a)]$. The graph  $\Gamma_1$ is symmetric with respect to the vertical axis, so $\Gamma_2$ is symmetric with respect to the diagonal.
Hence, the function $\vphi$ is an involution.
We will find a formula for $\vphi$ in terms of $\psi$.

If $y=\psi(x)$ then $(x,y)\in\Gamma_1$ and, therefore,
$(x+y, -x+y)\in\Gamma_2$. We will denote this point $ (u, \vphi(u))$. We have
\begin{align*}
u &=  x+y =  x+\psi(x),\\
\psi(x) &= u-x.
\end{align*}
Recall that $u\to (u+\psi(u))^{-1}$ denotes the inverse of the function $x \to x+\psi(x)$. Then $x= (u+\psi( u))^{-1}$ and
\begin{align*}
\vphi(u)&= -x+y =  -x+\psi(x) = -x +  u -x = - 2 x +  u
= -2 ( u+\psi( u))^{-1} +  u.
\end{align*}
\end{proof}

\section{Examples}\label{examp}

\begin{example}
Our first example is just a special case of Proposition \ref{m19.3}.
\bigskip

Consider a finite interval $(x_1, x_2)$ and a strictly increasing function $\vphi:(x_1,x_2)\to \R$. Let $y_1=\vphi(x_1)$ and $y_2 = \vphi(x_2)$. Suppose that $(x_1,x_2) \cap (y_1,y_2) =\emptyset$ and let $\vphi(y)=\vphi^{-1}(y)$ for $y\in(y_1,y_2)$. Then  $F(x) := \vphi(x) -x$  solves \eqref{m12.1} on $(x_1,x_2) \cup (y_1,y_2)$.

Alternatively, consider a finite interval $(x_1, x_2)$ and a strictly decreasing function $\vphi:(x_1,x_2)\to \R$. Let $y_1=\vphi(x_1)$ and $y_2 = \vphi(x_2)$. If $(x_1,x_2) \cap (y_2,y_1) \ne \emptyset$ then assume that $\vphi(y)=\vphi^{-1}(y)$ for $y\in(x_1,x_2) \cap (y_2,y_1)$. Let $\vphi(y)=\vphi^{-1}(y)$ for $y\in(y_2,y_1)$. Then  $F(x) := \vphi(x) -x$  solves \eqref{m12.1} on $(x_1,x_2) \cup (y_2,y_1)$.
\end{example}

\begin{example}\label{m29.1}
We will apply Proposition \ref{m19.3} and Lemmas \ref{m19.5} and \ref{m19.4} to construct a large family of analytic solutions $F$ on any interval $(a,b)$.
Let $d = (b-a)/2$.
There are many analytic functions $\psi:[-d,d]\to \R$ that are even, such that $\psi(d)=0$, and $|\psi(x)| <1$ for all $x\in (-d,d)$.
By Lemma \ref{m19.5},
\begin{align*}
\vphi(u) &:= u- 2( u+\psi( u))^{-1}
\end{align*}
is an involution on the interval $(-d,d)$. By Proposition \ref{m19.3}, the function $F(x) = \vphi(x) -x$ is a solution to \eqref{m12.1} on $(-d,d)$. Lemma \ref{m19.4} implies that
$F_1(x) := F(-d-a+x)$ is an analytic solution to \eqref{m12.1} on $(a,b)$.

\end{example}

\begin{example}
Corollary \ref{m19.2} and Example \ref{m29.1} show that one can decompose the real line into an arbitrary family of disjoint (open or closed) intervals or pairs of intervals (these could be points or pairs of points), and then construct a solution to \eqref{m12.1} on the whole real line by {\it{patching}} together solutions on separate intervals or pairs of intervals.

\end{example}

We now present some explicit examples for illustration or for future reference.

\begin{example}\label{m23.22}
The only straight lines symmetric with respect to the diagonal are the diagonal itself and lines orthogonal to the diagonal. Proposition \ref{m19.1} implies that the only linear solutions to \eqref{m12.1} are $F(x) \equiv 0$ and any function of the form $F(x) = - 2 (x+a)$, for any constant $a$. In all of these cases $A=\R$.

\bigskip

Alternatively, we could start with the involution $\varphi (x)=-x$, construct a solution to \eqref{m12.1} using Proposition \ref{m19.3} by letting $F(x)=-x+\varphi (x)=-2x$, and then use Lemma \ref{m19.4} to generate the solution $F_{a}^{+}(x)= - 2 (x+a)$.

\end{example}

\begin{example}\label{m23.24}
The piecewise constant function
\[
F(x)=(-1)^{k}\text{ for }k<x\leq k+1, \ k\in \Z,
\]
solves \eqref{m12.1} with $A=\R$.
One can use Lemma \ref{m19.4} to construct variants of this example.
\end{example}

\begin{example}
 Suppose that $c\neq 0$ is a constant and let  $F(x)=-(x+c/x)$
for $x\neq 0$, and $F(0)=0$. Then, for $x\neq 0$,
\[
F(x+F(x))=F(-c/x)=-(-c/x+c/(-c/x))=x+c/x=-F(x),
\]%
so $F$ is a solution to \eqref{m12.1}.

For $c>0$, Proposition \ref{m19.1} applies with $I = (0,\infty)$ and $J=(-\infty,0)$.

For $c<0$, Proposition \ref{m19.1} applies with
$I = (0,\sqrt{-c}]$ and $J=[\sqrt{-c}, \infty)$. It also applies with
$I = (-\infty,-\sqrt{-c}]$ and $J=[-\sqrt{-c}, 0)$.

\bigskip

We can also apply Proposition \ref{m19.3} and Lemma \ref{m19.4} to show that $F$ is a solution to \eqref{m12.1}.
If $c>0$, let $a= c^{-1/2}$,
$ \varphi (x)= -x^{-1}$ for $x\ne 0$,  $\varphi(0)=0$, and $F(x)=-x+\varphi (x)=-(x+ x^{-1})$.
Then $F_{a}^{\times }(x)=-a^{-1}(ax+ (ax)^{-1})=-(x+ cx^{-1})$ is a solution to \eqref{m12.1}.

If $c<0$, apply an analogous argument with
$a= |c|^{-1/2}$ and $ \varphi (x)= x^{-1}$ for $x\ne 0$.

\end{example}

\begin{example}\label{m23.21}
Here is a similar example embracing two earlier ones.
It is elementary to check that if $\Delta := b^{2}-4ac\neq 0$, then the function
\begin{align*}
G(x) = -  \frac{2(a x^2 + b x + c)}{2 a x + b}, \qquad x \ne -b/(2a),
\end{align*}
with $G(-b/(2a))=0$, satisfies \eqref{m12.1}.

For $a=0$ and $b\neq 0,$  $G(x)=F_{d}^{+}(x)$ for $F(x)=-2x$ and $d=c/b.$

For $a\neq 0$ and $\Delta>0$,  $G(x) =F_{d}^{\times }(x+b/2a)$ for $%
F(x)=-(x+ x^{-1})$ and $ d=|4a^{2}/\Delta|^{1/2} $.

For $a\neq 0$ and $\Delta<0$,  $G(x) =F_{d}^{\times }(x+b/2a)$ for $%
F(x)=-(x- x^{-1})$ and $ d=|4a^{2}/\Delta|^{1/2} $.

\end{example}

\section{PDEs with freezing}\label{sec:wave}

The main result of this section is the following.

\begin{theorem}\label{m23.1}
Suppose that $a>0$, $h:[-a,a]\to \R$ is  strictly increasing, differentiable, and odd.
Suppose that $T: [-a,a]\to \R$ is  even, differentiable, strictly decreasing on $[0,a]$, and such that $\left|\frac \prt {\prt x} T(x)\right| < 1$ for all $x\in(-a,a)$. Then
there exist continuous functions $\sigma:[-a,a]\times \R\to [0,\infty)$ and $\mu:[-a,a]\times \R\to \R$  satisfying
\begin{align}\label{n21.1}
\frac{\prt}{\prt t}\mu(x,t) &= - \frac{\prt}{\prt x}\sigma(x,t),\\
\frac{\prt}{\prt t}\sigma(x,t) &=  - \frac{\prt}{\prt x} \mu(x,t) ,\label{n21.2}
\end{align}
on the  set $\{(x,t) \in (-a,a)\times \R : \sigma(x,t)>0\}$, and such that
for $x\in[-a,a]$,
\begin{align}\label{m21.1}
 \inf\{t\geq 0: \sigma(x,t) = 0\}&= T(x),\\
\mu(x,T(x)) & = h(x).\label{m21.2}
\end{align}

\end{theorem}

\begin{remark}\label{a1.1}
(i)
The equations \eqref{n21.1}-\eqref{n21.2} arose in the analysis of the hydrodynamic limit of a system of pinned billiard balls on a line. See \cite{fold} for the definition of the pinned billiard balls model.
Heuristically speaking, after rescaling the position and time for the set of $n$ pinned balls and taking the limit as $n\to \infty$, we obtain a continuum family of infinitesimal pinned balls with velocities
$$v(x,t) =\mu(x,t) + \sigma(x,t) W(x,t),$$
where $W$ is space-time white noise.
The function $\mu(x,t)$ represents the mean velocity of the infinitesimal ball at position $x$ at time $t$, and $\sigma(x,t)$ represents its standard deviation.
The equations \eqref{n21.1}-\eqref{n21.2} are the result of informal calculations at the discrete level and passage to the limit of the rescaled system.

(ii)
Theorem \cite[Thm. 3.1]{fold} shows that the evolution of a finite system of pinned billiard balls must terminate after a finite number of collisions. 
In other words, the system
must  completely {\it{freeze}} at each location $x$ at some  time (depending on $x$).  
In terms of the hydrodynamic limit, this means that $\sigma(x,t)$ must be identically equal to $0$ for some $t_\infty<\infty$ and all $t\geq t_\infty$.
According to \eqref{n21.1}, $\mu(x,t)$ will not change after time $t_\infty$. These two phenomena are formally represented by  \eqref{m21.1}-\eqref{m21.2}.

(iii)
It follows easily from
\eqref{n21.1}-\eqref{n21.2}
that $\sigma$ satisfies the wave equation
\begin{align}\label{m20.1}
\frac{\prt^2}{\prt t^2}\sigma(x,t) &=   \frac{\prt^2}{\prt x^2} \sigma(x,t),
\end{align}
on the  set $\{(x,t): \sigma(x,t)>0\}$.

Let $X$ be the inverse of $T$ on $[0,a]$, i.e., $X(t) = x$ if $T(x) =t$. Since $\left|\frac \prt {\prt x} T(x)\right| < 1$, we have $\left|\frac \prt{\prt t} X(t)\right| >1$. This seemingly contradicts the bound on the wave speed  which is normalized to be 1 in \eqref{m20.1}. However, there is no contradiction because the boundary of the frozen region does not represent a physical signal that could be used to send information, so it is  not limited by the wave speed.

(iv)
Key to Theorem \ref{m23.1} are solutions to the functional equation%
\begin{align}\label{m23.13}
G(x-2G(x))=-G(x)
\end{align}
(see \eqref{m23.11} below with $G=g^{-1}$), which is equivalent, after substituting $2x$ for $x$ and taking $F(x)=-G(2x)$, to
\[
F(x+F(x))=-F(x),
\]%
i.e., \eqref{m12.1}. For any solution $G$ to \eqref{m23.13} and any strictly increasing
odd function $h$ (the desired terminal velocity profile), the composition%
\[
f(x)=\frac{1}{2}h(G(x))
\]%
yields a familiar solution to the wave equation satisfied by $\sigma $,
with  forward and backward traveling waves with standardized velocity $c=1$:
\[
\sigma (x,t)=f(x+ct)+f(-x-ct).
\]%
In this sense, the functional equation \eqref{m23.13}  generates
the fundamental solution to the PDE\ system. In the proof below, $G$ is the
inverse function $g^{-1}$ to $g(x)=x-T(x),$ where $T(x)$ is the desired
freezing time, satisfying the conditions of Theorem 4.1. Thus the freezing
time $T(x)$ determines and is determined by the ``generator'' $G.$ Indeed%
\begin{align}\label{m23.12}
T(x)=x-g(x)=x-G^{-1}(x).
\end{align}

(v)
Note that the shape $h$ of the frozen function $\mu$
can be chosen independently of the ``speed''  of  freezing represented by $T$.

(vi) It is easy to see that the proof applies also to the case when the finite interval $[-a,a]$ is replaced with $(-\infty, \infty)$.

\end{remark}

\begin{proof}[Proof of Theorem \ref{m23.1}]
In view of \eqref{m20.1}, we will look for
$\sigma$  having the form $\sigma(x,t) = f(x- t) + f(-x- t)$.
It is easy to check that the following functions solve \eqref{n21.1}-\eqref{n21.2},
\begin{align}\label{m23.2}
\sigma(x,t) &= f(x- t) + f(-x- t),\\
\mu(x,t) &= f(x- t) - f(-x- t) ,\label{m23.3}
\end{align}
for any differentiable function $f$.

Let $z=g(x) =x-T(x)$ and note that $g$ is strictly increasing, by the assumption that $\left|\frac \prt {\prt x} T(x)\right| < 1$.  Hence the inverse function $g^{-1}$ is well defined and accordingly $x=g^{-1}(z)$.
Let
\begin{align}
f(z)&=\frac 1 2 h(g^{-1}(z)).\label{m23.5}
\end{align}
This defines the function $f(z)$ for
\begin{align}\label{m26.1}
z\in g([-a,a]) = [-a-T(-a), a-T(a)]= [-a-T(a), a-T(a)].
\end{align}
We extend $f$ to the whole real line so that it remains differentiable but is otherwise arbitrary.

We will show that if we substitute this function $f$ in \eqref{m23.2}-\eqref{m23.3} then   $(\sigma, \mu)$ satisfies \eqref{m21.1}-\eqref{m21.2}.

We have
\begin{align}\label{m21.4}
f(x-T(x))& = f(z).
\end{align}
Since $T(x)= x-z =  g^{-1}(z)-z$,
\begin{align*}
& -x-T(x)=  -g^{-1}(z) - g^{-1}(z)+z = z - 2  g^{-1}(z),
\end{align*}
and, therefore,
\begin{align}\label{m23.8}
f(-x-T(x)) = f(z - 2  g^{-1}(z)).
\end{align}
Note that if $x\in[-a,a]$, then
\begin{align*}
-x-T(x) \in [-a-T(a), a-T(-a)] = [-a-T(a), a-T(a)]
\end{align*}
and so, in view of \eqref{m26.1}, the function in \eqref{m23.8} is defined by the formula given in \eqref{m23.5}.

Let $z\to( z-T( z))^{-1}$ denote the inverse of the function $x\to x- T(x)$.
Inverting $g(x)=  x - T(x)$ yields $g^{-1}(z) =   ( z-T( z))^{-1}$.
The function $\vphi(z) := z-2( z-T( z))^{-1}=z-2g^{-1}(z)$ is an involution on the interval
$[-a-T(-a), a-T(a)]=[-a-T(a), a-T(a)]$
according to Lemma \ref{m19.5} and \eqref{m23.7}, because $-T$ satisfies the assumptions placed on $\psi$ in that lemma.
By Proposition \ref{m19.3}, $F(z) := \vphi(z) - z$ is a solution to \eqref{m12.1}, i.e.
\begin{align}\label{m23.9}
F(z) = -  F(z + F(z)).
\end{align}
We have
\begin{align*}
F(z) =  \vphi(z) - z = -2( z-T( z))^{-1} = -  2 g^{-1}(z),
\end{align*}
so \eqref{m23.9} gives
\begin{align*}
-  2 g^{-1}(z) &=   2 g^{-1}(z -  2 g^{-1}(z)).
\end{align*}
That is, as in \eqref{m23.13},
\begin{align}\label{m23.11}
 g^{-1}(z -  2 g^{-1}(z))=-g^{-1}(z).
\end{align}

Since $h$ is assumed to be odd and strictly increasing, and $f$ is defined by \eqref{m23.5},
\begin{align*}
 h(g^{-1}(z)) &=  h(-g^{-1}(z - 2 g^{-1}(z)))  =- h(g^{-1}(z -  2 g^{-1}(z))),\\
 f(z) &=   - f(z -  2 g^{-1}(z)).
\end{align*}
Substituting from \eqref{m21.4} and \eqref{m23.8} gives
\begin{align}\label{m21.3}
f(x-T(x)) = - f(-x-T(x)),
\end{align}
so
\begin{align*}
 \sigma(x,T(x)) = f(x- T(x)) + f(-x-T(x))=0.
\end{align*}

Recall that $g$ is strictly increasing and continuous, so the same holds for $g^{-1}$. This also applies to $h$, by assumption. Hence, the formula \eqref{m23.5} shows that $f$  is strictly increasing and continuous. This implies that for a fixed $x$, the functions $t\to f(x-t)$ and $t\to f(-x-t)$ are strictly decreasing, and, therefore, $t\to \sigma(x,t)$ is strictly decreasing. Since $\sigma(x,T(x)) =0$, we conclude that $\sigma(x,t) >0 $ for $t< T(x)$.
This proves \eqref{m21.1}.

By \eqref{m23.5}, \eqref{m21.4} and \eqref{m21.3}
\begin{align*}
 \mu(x,T(x)) &= f(x-T(x)) - f(-x-T(x))
= 2 f(x-T(x)) = 2 f(z) = h(g^{-1}(z))\\
& = h(x) .
\end{align*}
This proves \eqref{m21.2}.

We restrict formulas \eqref{m23.2}-\eqref{m23.3} defining $\sigma$ and $\mu$ to $(x,t)$ with $t < T(x)$. For $t\geq T(x)$, we put
$$\sigma(x,t) = 0 \hbox{ and } \mu(x,t) = h(x).$$
\end{proof}

\begin{remark}
We note that the involution $\varphi $
arising from the function $F$ in the proof of Theorem \ref{m23.1} (see \eqref{m23.9}) is a ``reflection in $-T(x)$,'' i.e.,
\[
\varphi (-T(x)\pm x)=-T(x)\mp x.
\]
Indeed,%
\begin{equation}\label{m23.14}
\varphi (-T(x)+x)=\varphi (z)=z-2g^{-1}(z)=x-T(x)-2x=-T(x)-x.
\end{equation}
From here, as $\varphi $ is an involution,
\begin{equation}\label{m23.15}
-T(x)+x=\varphi \varphi (-T(x)+x)=\varphi (-T(x)-x).
\end{equation}%
\end{remark}

We now illustrate Theorem \ref{m23.1} with three explicit examples. From the purely mathematical point of view, the examples are routine applications of Theorem \ref{m23.1}. However, in the context of the pinned balls model discussed in Remark \ref{a1.1} (i), it is far from obvious how the freezing might proceed. Does it have to be progressive? Or can it happen simultaneously
along a substantial stretch of the space? Our examples show that three possible scenarios can occur: continuous progression of the front of the frozen region, simultaneous freezing over the whole spacial interval at the terminal time, and a mixture of the two phenomena---progressive freezing followed by the simultaneous freezing of a part of the space at the terminal time.

Below, the notation is as in the statement and proof of Theorem \ref{m23.1}.

\begin{example}\label{m30.1}
Let
\begin{align*}
a& = \sqrt{4 -1/e^4},\\
T(x) & = 2- \sqrt{x^2+1/e^4},\\
h(x) &=  2 \log \left(\sqrt{ x^2+1/e^4}+x\right)+4 .
\end{align*}
Then for $t\geq 0$ and $x\in[-a,a]$,
\begin{align*}
f(x) &= \log(x+2) + 2,\\
\sigma(x,t) &=  \log(x-t+2)  + \log(-x-t+2) + 4,\\
\mu(x,t) &=\log(x-t+2) - \log(-x-t+2) ,\\
\sigma(x,0) &=
\log (2-x)+\log (x+2)+4,\\
\mu(x,0)
&=\log (x+2)-\log (2-x).
\end{align*}
It follows from \eqref{m23.12}, the formula $F(x)=-G(2x)=-g^{-1}(2x)$, and some calculations, that
\begin{align*}
F(x) = -  \frac{(2x+2)^2 -e^{-4}}{4x + 4}, \qquad x \ne -1.
\end{align*}
This is a special case of Example \ref{m23.21}.

\end{example}

\begin{example}\label{m30.2}

We will consider the problem on the spacial interval $(-a,a)=(-1,1)$.
We take
\begin{align*}
f(x) =
\begin{cases}
1+x, & x\leq 0,\\
1-x, & x>0.
\end{cases}
\end{align*}
Then for $x\in(-1,1)$ and $t\geq 0$,
\begin{align*}
\sigma(x,t) &=
\begin{cases}
2-2|x|, & |x|>  t,\\
2-2 t,  & |x|\leq  t,
\end{cases}\\
\sigma(x,0) &=  2-2|x|,\\
\mu(x,t) &=
\begin{cases}
2 t, & |x|>  t,\\
2x , & |x|\leq  t,
\end{cases}\\
\mu(x,0) &= 0,\\
T(x) & = 1,\\
h(x) & = 2x.
\end{align*}
In the present example, freezing is simultaneous at the terminal time, and the terminal shape $h$ is a linear function.

We note that this case arises from the solution $F(x)=-2(x+1/2)$ of \eqref{m12.1}, as in Example \ref{m23.22}.

\end{example}

\begin{example}\label{m23.23}

We will consider the problem on the spacial interval $(-a,a)=(-2,2)$.
We take $h(x) \equiv 1$ and
\begin{align*}
T(x) =
\begin{cases}
1-|x|/2, & 1\leq |x| \leq 2,\\
1/2, & |x|<1.
\end{cases}
\end{align*}
Then
\begin{align*}
g^{-1}(z) =
\begin{cases}
\frac 2 3 (z+1), & z \geq 1/2,\\
z+1/2, & -3/2 < z < 1/2 ,\\
2(z+1), & z\leq -3/2.
\end{cases}
\end{align*}
From here it is straightforward to check that this corresponds to a solution $F$ of \eqref{m12.1} in which $\varphi$ is an involution taking $[-1,-3/4]$ to $[1/4,1]$ and $[-3/4,1/4]$ to itself.
\bigskip

The following formulas hold for $|x|\leq2$, $0\leq t \leq 1/2$, $t\leq 1-|x|/2$,
\begin{align*}
\sigma(x,t)&=
\begin{cases}
\frac 2 3 x -\frac 4 3 t + \frac 4 3, &
x + \frac 3 2 \leq t \leq -x -\frac 1 2,\\
\frac 16 x - \frac56 t +\frac 7{12}, &
x-\frac 12 \leq t \leq x+\frac 32 \text{ and } t \leq - x - \frac12,\\
\frac 1 2 - t ,&
 x-\frac 12 \leq t \leq x+\frac 32 \text{ and }
- x - \frac 12 \leq t \leq -x + \frac 32,\\
-\frac 16 x - \frac56 t +\frac 7{12}, &
 t\leq x-\frac 12  \text{ and }
- x - \frac 12 \leq t \leq -x + \frac 32,\\
-\frac 2 3 x -\frac 4 3 t + \frac 4 3, &
-x+\frac 32 \leq t \leq x - \frac 12,
\end{cases}\\
\mu(x,t)&=
\begin{cases}
\frac 4 3 x -\frac 2 3 t + \frac 2 3, &
x + \frac 3 2 \leq t \leq -x -\frac 1 2,\\
\frac 56 x - \frac16 t -\frac 1{12}, &
x-\frac 12 \leq t \leq x+\frac 32 \text{ and } t \leq - x - \frac12,\\
 x,&
 x-\frac 12 \leq t \leq x+\frac 32 \text{ and }
- x - \frac 12 \leq t \leq -x + \frac 32,\\
\frac 56 x + \frac16 t +\frac 1{12}, &
 t\leq x-\frac 12  \text{ and }
- x - \frac 12 \leq t \leq -x + \frac 32,\\
\frac 4 3 x +\frac 2 3 t - \frac 2 3, &
-x+\frac 32 \leq t \leq x - \frac 12.
\end{cases}
\end{align*}

\begin{align*}
\sigma(x,0)&=
\begin{cases}
\frac 2 3 x  + \frac 4 3, &
x  \leq  -\frac 3 2,\\
\frac 16 x  +\frac 7{12}, &
-\frac 3 2 \leq x \leq  - \frac12,\\
\frac 1 2  ,&
 -\frac 12 \leq x  \leq  \frac 12,\\
-\frac 16 x  +\frac 7{12}, &
\frac 12 \leq x \leq \frac 32,\\
-\frac 2 3 x  + \frac 4 3, &
x \geq  \frac 32,
\end{cases}\\
\mu(x,0)&=
\begin{cases}
\frac 4 3 x  + \frac 2 3, &
x  \leq  -\frac 3 2,\\
\frac 56 x  -\frac 1{12}, &
-\frac 3 2 \leq x \leq  - \frac12,\\
 x,&
-\frac 12 \leq x  \leq  \frac 12,\\
\frac 56 x  +\frac 1{12}, &
\frac 12 \leq x \leq \frac 32,\\
\frac 4 3 x  - \frac 2 3, &
x \geq  \frac 32.
\end{cases}
\end{align*}

In this example, we have mixed modes of freezing---continuous freezing at the outer regions of the spacial interval $(-2,2)$ and simultaneous freezing in the middle of the interval.

The functions  $T$ and $h$ are not differentiable, hence, the assumptions of Theorem \ref{m23.1} are not satisfied. But one can justify the example by considering differentiable approximations of $T$ and $h$ and passage to the limit. Alternatively, one can check directly that \eqref{n21.1}-\eqref{n21.2} and \eqref{m21.1}-\eqref{m21.2} are satisfied.
\end{example}

 \begin{remark}
  In Example \ref{m23.23}, underlying the passage of a monotone sequence of differentiable solutions $F_{n}$ of \eqref{m12.1}
   to a continuous pointwise limit solution of \eqref{m12.1} is the convergence of the corresponding involutions $\varphi_n$ and implicit reliance on uniformity, appealing to the Dini/P{\'o}lya--Szeg\H{o} monotone convergence theorems. See respectively \cite[7.13]{Rud1976} for monotone convergence of continuous functions to a continuous pointwise limit, and \cite{PolS1925}, Vol. 1 pp. 63, 225, Problems II 126 and 127, or  \cite[§17, pp. 104--5]{Boa1981}, when the functions are monotone.

  \end{remark}

\section{Acknowledgments}

We are grateful to Nicholas Bingham, Jeremy Hoskins, Christopher Small, Stefan Steinerberger and John Sylvester for very helpful advice.


\section{Addenda}
In this section we include some extra observations on the functional equation \eqref{m12.1}. For the first result cf. Remark \ref{m23.20} (i).

\begin{proposition}If a solution $F$ to \eqref{m12.1} on $(w,\infty)$  has a  vertical asymptote at $x=w$ and $\lim_{x\rightarrow w+}F^{\prime }(x)=-\infty $, then
\[
\lim_{x\rightarrow +\infty }F^{\prime }(x)=-1.
\]
\end{proposition}

\begin{proof}
Recall Remark \ref{m23.20} (i). The function $\vphi(x) := x+F(x)$ is an involution so its graph is symmetric with respect to the diagonal. We have assumed that $\lim_{x\rightarrow w+}F^{\prime }(x)=-\infty $ so $\lim_{x\rightarrow w+}\vphi^{\prime }(x)=-\infty $. Hence, by symmetry,
$\lim_{x\to+\infty}\vphi^{\prime }(x)=0 $, and, therefore,
$\lim_{x\rightarrow +\infty }F^{\prime }(x)=
\lim_{x\to+\infty}(\vphi(x)-x)' = -1$.
\end{proof}

\bigskip

The result is unchanged if instead $\lim_{x\rightarrow w+}F^{\prime }(x)=+\infty $.
\bigskip

Applying reflection (see Lemma \ref{m19.4}),
an analogous calculation with a leftmost vertical asymptote supports the corresponding conclusion of an asymptotic slope of $-1$ at $-\infty$.

\bigskip

\begin{proposition} \textit{If the solution function }$F$\textit{\ is
real analytic near }$x=0$\textit{\ and }$F(0)=0,$\textit{\ then either }$%
F^{\prime }(0)=-2$\textit{\ or }$F\equiv 0$ \textit{near} $x=0.$
\end{proposition}
\bigskip

\begin{proof} Differentiating the functional equation for $F$ we obtain%
\[
0=F^{\prime }(x)+F^{\prime }(x+F(x))(1+F^{\prime }(x)),
\]%
so setting $x=0$%
\[
0=F^{\prime }(0)+F^{\prime }(0)(1+F^{\prime }(0))=F^{\prime }(0)(2+F^{\prime
}(0)).
\]%
So either $F^{\prime }(0)=-2$ or $F^{\prime }(0)=0.$

A little experimentation shows that for $n\geq 1$%
\begin{align}\label{a10.1}
0=F^{(n)}(x)+F^{(n)}(x+F(x))(1+F^{\prime
}(x))^{n}+\sum_{j=1}^{n-1}F^{(j)}(x+F(x))G_{j}^{n}(x),
\end{align}
for some real analytic functions $G_{j}^{n}(x).$ We will prove this formula
by induction. It is easy to see that it is valid for $n=1$. Assuming that it holds for $n$, we obtain by differentiating
\begin{align*}
0 &=F^{(n+1)}(x)+F^{(n+1)}(x+F(x))(1+F^{\prime
}(x))^{n+1}\\
&\qquad+nF^{(n)}(x+F(x))(1+F^{\prime }(x))^{n-1}F^{\prime \prime }(x) \\
&\qquad+\sum_{j=1}^{n-1}F^{(j+1)}(x+F(x))(1+F^{\prime
}(x))G_{j}^{n}(x)+\sum_{j=1}^{n-1}F^{(j)}(x+F(x))(G_{j}^{n})^{%
\prime }(x)
\end{align*}%
and this is of the form%
\[
0=F^{(n+1)}(x)+F^{(n+1)}(x+F(x))(1+F^{\prime
}(x))^{n+1}+\sum_{j=1}^{n}F^{(j)}(x+F(x))G_{j}^{n+1}(x),
\]
so \eqref{a10.1} is proved.

Assume that $F(0)=0$ and $F'(0)=0$. We will prove that $F(x) \equiv 0$ in a neighborhood of $0$ by showing that all derivatives of $F$ at 0 vanish.

Suppose that $F^{(j)}(0)=0$ for
$0\leq j\leq n-1$. Then
setting $x=0$ in \eqref{a10.1}, it follows that
\[
0=F^{(n)}(0)+F^{(n)}(0)+\sum_{j=1}^{n-1}F^{(j)}(0)G_{j}^{n}(0)=2F^{(n)}(0)
\]%
and so also $F^{(n)}(0)=0.$ This implies that $F\equiv 0.$ \hfill
\end{proof}

\end{document}